\documentclass[a4paper,10pt,reqno]{amsart}%
\usepackage[utf8]{inputenc}
\usepackage[T1]{fontenc}
\usepackage{amsmath}%
\usepackage{amsfonts}%
\usepackage{txfonts}
\usepackage{amssymb}%
\usepackage{graphicx}
\usepackage{caption}
\usepackage{subcaption}
\usepackage{bbm}
\usepackage{stmaryrd}
\usepackage[paper=a4paper,left=28mm,right=28mm,top=30mm,bottom=30mm]{geometry}
\usepackage{epsfig}
\usepackage{here}
\usepackage{multirow}
\usepackage{marvosym,wasysym}
\usepackage{fancyhdr}
\usepackage{hyperref}
\usepackage{natbib}

\usepackage{mathtools}

\theoremstyle{plain}
\newtheorem{theorem}{Theorem}[section]

\newtheorem{lemma}[theorem]{Lemma}

\theoremstyle{definition}

\theoremstyle{remark}

\numberwithin{equation}{section}

\usepackage{bbm}

\newcommand{\N}{\mathbbm{N}}
\newcommand{\R}{\mathbbm{R}}

\newcommand{\Rplus}{\mathbbm{R}_{\geq 0}}
\newcommand{\p}{\mathbbm{P}}

\newcommand{\cA}{\mathcal{A}}
\newcommand{\cB}{\mathcal{B}}

\newcommand{\cG}{\mathcal{G}}
\newcommand{\cF}{\mathcal{F}}

\newcommand{\cN}{\mathcal{N}}

\newcommand{\expectation}[1]{\mathbbm{E}\left [ \, #1 \, \right ]}
\newcommand{\Var}{\mathrm{Var }}

\renewcommand{\epsilon}{\varepsilon}
\renewcommand{\phi}{\varphi}

\newcommand{\pspace}{(\Omega,\cA,\p)}

\newcommand{\intd}[1]{\,\mathrm{d}#1}

\newcommand{\norm}[1]{\left\lVert #1 \right\rVert}
\newcommand{\scalar}[2]{\left\langle #1,#2 \right\rangle}

\allowdisplaybreaks

\begin{document}

\author[Johannes T. N. Krebs]{Johannes T. N. Krebs}
\thanks{The author gratefully acknowledges the financial support of the Fraunhofer ITWM which is part of the Fraunhofer-Gesellschaft zur F\"{o}rderung der angewandten Forschung e.V}
\email{krebs@mathematik.uni-kl.de}
\address{University of Kaiserslautern, Erwin-Schr{\"o}dinger-Strasse, 67663 Kaiserslautern}
\title{Consistency and Asymptotic Normality of Stochastic Euler Schemes for Ordinary Differential Equations}
\date{December, 1 2015}

\begin{abstract}
General stochastic Euler schemes for ordinary differential equations are studied. We give proofs on the consistency, the rate of convergence and the asymptotic normality of these procedures.
\end{abstract}

\maketitle

\section{Introduction}\label{Introduction}
We study the consistency and asymptotic normality of stochastic Euler schemes which are designed to approximate ordinary differential equations. Euler schemes are often used to simulate stochastic differential equations. \cite{fierro2001euler} study the consistency of these schemes in the context of It\^{o} stochastic differential equations. However, this idea can be used to approximate ordinary differential equations, too: \cite{fierro2007stochastic} consider a special kind of Euler approximation for a given ODE. In this paper, we generalize the idea: Let there be given the ODE system $\dot{x} = F(t, x)$, $x(0) = x_0\in\R^d$, $t\in[0,T]$, $0<T<\infty$. Then we approximate the solution $x$ on a partition $\pi^N$ of $[0,T]$ with a stochastic Euler scheme that is based on random variables $\tilde{F}^N_k$ instead on $F$. This approach can be useful in applications where one aims at approximating the trajectory of such a solution $x$ for a function $F$ which is costly to evaluate, for instance, in the case where $F$ is the sum of (finitely) many single functions $f_i$, $i\in I$, i.e. $F = \sum_{i\in I} f_i$. The paper is organized as follows: In Section~\ref{Preliminaries} we introduce the basic notions and regularity conditions of the model. In Section~\ref{ConsistencyAndRateOfConvergence} we give consistency results for our general Euler scheme. We state results on the asymptotic normality of the procedure in Section~\ref{AsymptoticNormality}. Appendix~\ref{Appendix} contains some background material.

\section{Preliminaries}\label{Preliminaries}
We denote for $p\ge 1$ by $\norm{\,\cdot\,}_{p}$ the $p$-norms on the $d$-dimensional Euclidean space. Let $T\in\R_+$ be a finite time horizon and let $F=(F_1,\ldots,F_d)': [0,T] \times \R^d \rightarrow \R^d$ be a continuous vector valued function. $F$ fulfills the following growth conditions w.r.t. the first and second coordinate for $s,t \in [0,T]$ and for $x,y\in\R^d$
\begin{align}
		&  \norm{ F(s,x) - F(t,x) }_1 \le K_1 \left(1+ \norm{x}_1 \right) |s-t|, \label{growthTime}\\
		& \norm{F(t,x)-F(t,y)}_1 \le K_2\, ||x-y||_1, \label{growthSpace}
		\end{align}
where $0<K_1,K_2 < \infty$ are some positive constants. Let there be given the ODE $\dot{x} = F(t,x)$ and $x(0) = x_0\in\R^d$ on $[0,T]$. Denote the unique global solution of this equation by $x: [0,T] \rightarrow \R^d,\, x(t) := x(0) + \int_0^t F\left(s, x(s) \right) \intd{s}$. This solution is guaranteed by the global Lipschitz condition \eqref{growthSpace} and is Lipschitz-continuous with a Lipschitz-constant $0<C<\infty$, i.e. $\norm{x(s)-x(t)}_1 \le C |s-t|$. Next, choose a sequence of partitions, $\pi^N$, $N \in \N_+$, of the interval $[0, T]$ such that $\pi^N$ consists of the points $t_0^N=0 < t_1^N < \ldots < t_{K_N}^N = T$ and such that the mesh of the partition $\Delta^N := \max_{1 \le k \le K_N} \Delta^N_k$ converges to zero as $N\rightarrow\infty$, where $\Delta^N_k := t^N_k - t^N_{k-1}$. The stochastic part is introduced via a probability space $\pspace$ endowed with the following mappings: For $N \in \N_+$ and $k=1,\ldots,K_N$ the function
\[
	\tilde{F}^N_k = (\tilde{F}^N_{k,1},\ldots, \tilde{F}^N_{k,d} )' : [0,T]\times \R^d \times \Omega \rightarrow \R^d \text{ is measurable}\left[\cB\left([0,T]\times\R^d\right) \otimes \cA \, \big|\, \cB\left(\R^d\right)\right]
\]
and is Lipschitz-continuous w.r.t. the second coordinate with the same Lipschitz constant as $F$.\\
Furthermore, for any selection of time-space coordinates $(t_1,y_1),\ldots,(t_{K_N},y_{K_N}) \in [0, T]\times\R^d$ the random variables $\tilde{F}^N_1(t_1,y_1), \ldots, \tilde{F}^N_{K_N}(t_{K_N},y_{K_N})$ are independent and each $\tilde{F}^N_k$ is an unbiased estimator of $F$ in the sense that $\expectation{ \tilde{F}^N_k (t, y) }= F(t,y) $ for $(t,y)\in [0,T]\times\R^d$. In addition, we assume that there exists a constant $0<K_3<\infty$ such that for each $t\in [0,T]$, $a=1,\ldots,d$, $k=1,\ldots,K_N$ and $N\in\N_+$ the variance of the approximation is bounded as	$\text{Var}\left[ \tilde{F}^N_{k,a} (t,x(t)) \right] \le K_3$. We generate for each $N\in\N_+$ a stochastic sequence $\hat{x}^N = \left\{ \hat{x}^N \left( t^N_{i} \right): i=0,\ldots, K_N\right\}$ according to the rule
\begin{align}\begin{split}\label{defXHat0}
			\hat{x}^N\left(t^N_0\right) := x(0) \in \R^d \text{ and }	\hat{x}^N \left(t^N_{i} \right)  := \hat{x}^N \left(t^N_{i-1} \right) + \Delta^N_{i} \tilde{F}^N_i \left(t^N_{i-1}, \hat{x}^N \left(t^N_{i-1} \right) \right) \text{ for } i = 1, \ldots , K_N.
\end{split}\end{align}
We pass from this sequence $\left\{ \hat{x}^N \left( t^N_{i} \right): i=0,\ldots, K_N \right\}$ to a right-continuous process which we denote again by $\hat{x}^N$, namely, we define
\begin{align}
		\hat{x}^N (t) := \hat{x}^N \left(t^N_i\right) \text{ for } t \in \left[ t^N_i, t^N_{i+1} \right) \text{ for } i=0,...,K_N -1 \text{ and } \hat{x}^N(T) = \hat{x}^N \left( t^N_{K_N} \right). \label{defXHat}
\end{align}
In the following, when speaking of $\hat{x}^N$, we shall always refer to this c\`{a}dl\`{a}g process. Moreover, $\left\{ \cF^N (\,\cdot\,): N \in \N_+\right\}$ is a sequence of filtrations on $\pspace$ such that for each $N\in\N_+$ the filtration $\cF^N(\,\cdot\,)$ is the natural and right-continuous filtration of the process $\hat{x}^N$ from equation \eqref{defXHat}.

\section{Consistency and Rate of Convergence}\label{ConsistencyAndRateOfConvergence}
We come to the first main result of this paper, this is the convergence in mean of the processes $\hat{x}^N$, $N\in \N_+$, namely
\begin{theorem}[$L^2$-convergence of $\hat{x}^N$ to $x$] \label{L2convergence}
Let the sequence of stochastic processes $(\hat{x}^N: N\in\N_+)$ be defined in equations \eqref{defXHat0} and \eqref{defXHat}. Let $x$ be the unique global solution to the ordinary differential equation. Then, there exists a constant $0<B<\infty$ such that
\[
		\norm{ \sup_{t\in [0,T] } \norm{  \hat{x}^N (t) - x (t ) }_1 }_{L^2(\p)} \le B \sqrt{\Delta^N }.
\]
\end{theorem}

\begin{proof}
Throughout the proof we shall write $\norm{\,\cdot\,}$ for the Euclidean 1-norm on $\R^d$. Furthermore, we set $x^N_k := x^N\left(t^N_k\right)$ and $\hat{x}^N_k := \hat{x}^N \left(t^N_k\right)$, for $k=0,\ldots,K_N$. First, we consider $\hat{x}^N$ at the points $t^N_k$, $k=1,\ldots,K_n$. We derive for the difference $\hat{x}^N_k - x^N_k$ at each $k = 1,\ldots,K_N$ the equation
\begin{align*}
	\hat{x}^N_k - x^N_k  &= \hat{x}^N_{k-1} - x^N_{k-1} + \left\{ \tilde{F}^N_k\left(t^N_{k-1},\hat{x}^N_{k-1} \right) - \tilde{F}^N_k\left(t^N_{k-1},x^N_{k-1}\right) \right\} \Delta^N_k + \left\{ \tilde{F}^N_k\left(t^N_{k-1},x^N_{k-1}\right) - {F}\left(t^N_{k-1},x^N_{k-1}\right) \right\} \Delta^N_k \\
	&\quad+ \left\{ F\left(t^N_{k-1}, x^N_{k-1} \right) \Delta^N_k - \int_{\Delta^N_k} F\left(s, x(s) \right) \intd{s} \right\}.
\end{align*}
Successive iteration down to $t^N_0$ yields
\begin{align}\begin{split}\label{succItConXhat}
	\hat{x}^N_k - x^N_k  &= \sum_{j=1}^{k} \left\{ \tilde{F}^N_j\left(t^N_{j-1}, \hat{x}^N_{j-1}\right) - \tilde{F}^N_j\left(t^N_{j-1}, x^N_{j-1}\right) \right\} \Delta^N_j + \sum_{j=1}^k \left\{ \tilde{F}^N_j\left(t^N_{j-1}, x^N_{j-1}\right) - F\left(t^N_{j-1}, x^N_{j-1}\right) \right\} \Delta_j^N
\\
	&\quad+  \sum_{j=1}^k \left\{ F\left(t^N_{j-1}, x^N_{j-1}\right) \Delta^N_j - \int_{ \Delta_j^N} F(s,x(s)) \intd{s} \right\}.
\end{split}\end{align}
By the growth condition w.r.t. the time coordinate and the Lipschitz condition w.r.t. the space coordinate, we have for the last term in \eqref{succItConXhat}
\begin{align*}
		&\norm{ \int_{\Delta^N_j} F\left( t^N_{j-1}, x^N_{j-1} \right) - F(s,x(s)) \,\intd{s}	}
		\le \left\{ K_1 \left( 1+ \sup_{t \in [0,T]} \norm{x(t)} \right) + K_2 C \right\} \left(\Delta^N_j \right)^2.
\end{align*}
We put for short $L := K_1 \left( 1+ \sup_{t \in [0,T]} \norm{x(t)} \right) +K_2 C $. We can estimate the left-hand side of \eqref{succItConXhat} using the Lipschitz condition on the stochastic approximations $\tilde{F}^N_k$ of $F$ to arrive at the following bound for $\norm{ \hat{x}^N_k - x^N_k }$:
\begin{align}\label{preGronwall}
	\norm{ \hat{x}^N_k - x^N_k } \le K_2 \sum_{j=1}^k \norm{ \hat{x}^N_{j-1} - x^N_{j-1} } \Delta^N_j + \norm{ \sum_{j=1}^k \left( \tilde{F}^N_j \left(t^N_{j-1}, x ^N_{j-1} \right) - F\left( t^N_{j-1}, x^N_{j-1} \right) \right) \Delta_j^N		} + \sum_{j=1}^k L \left(\Delta^N_j \right)^2.
\end{align}
We apply the discrete Gronwall inequality from Lemma~\ref{discreteGronwallInequality} from Appendix~\ref{Appendix} to the bound given in \eqref{preGronwall}. We get
\begin{align}
\norm{ \hat{x}^N_k - x^N_k } &\le \norm{ \sum_{j=1}^k \left( \tilde{F}^N_j \left(t^N_{j-1}, x _{j-1}\right) - F\left(t^N_{j-1}, x^N_{j-1}\right) \right) \Delta^N_j } + \sum_{j=1}^k L \left( \Delta^N_j \right) ^2 \nonumber \\
&\quad + \sum_{j=1} ^{k-1} \left\{ \norm{ \sum_{i=1}^j \left( \tilde{F}^N_i\left(t^N_{i-1}, x^N_{i-1} \right) - F\left(t^N_{i-1}, x^N_{i-1} \right) \right) \Delta^N_i } + L \sum_{i=1}^j \left( \Delta^N_i \right) ^2 \right\} \left\{ K_2 \Delta^N_j \right\} \exp\left( \sum_{i=j+1}^{k-1} K_2 \Delta^N_i \right) \nonumber \\
\begin{split} \label{preDoob3}
	&\le	M_1 \Delta^N + \norm{ \sum_{j=1}^k \left( \tilde{F}^N_j \left(t^N_{j-1}, x^N_{j-1}\right) - F\left(t^N_{j-1}, x^N_{j-1}\right) \right) \Delta^N_j } \\
	&\quad + M_2  \left\{ \sum_{j=1}^{k-1} \norm{ \sum_{i=1}^j \left( \tilde{F}^N_i \left(t^N_{i-1}, x^N_{i-1} \right) - F\left(t^N_{i-1}, x^N_{i-1} \right) \right) \Delta^N_i }\Delta^N_j \right\}, \end{split}
\end{align}
where the constants $M_1$ and $M_2$ are given by $M_1 := L T + K_2 L T^2 \exp( K_2 T)$ and $M_2 := K_2 \exp(K_2  T )$. In particular, for $A:= 2M_1^2$ and $B:=2( 1+M_2 T)^2$ it holds good that
\begin{align}\label{preDoob}
		\sup_{1 \le j \le K_N} \norm{  \hat{x}^N_j - x^N_j }^2 &\le A\left(\Delta^N\right)^2 + B \sup_{1 \le j \le K_N} \norm{ \sum_{i=1}^j \left(\tilde{F}^N_i \left(t^N_{i-1}, x^N_{i-1} \right) - F\left(t^N_{i-1}, x^N_{i-1} \right) \right)	\Delta^N_i		}^2.
\end{align}
Next, we use the independence assumptions on the $\tilde{F}^N_k$ and the assumption that in each point $\expectation{ \tilde{F}^N_k(t,x)} = F(t,x) $. We show that the discrete process
\begin{align*}
	\left\{t^N_0,\ldots,t^N_{K_N} \right\} \ni t^N_j \mapsto \norm{ \sum_{i=1}^{K_n} 1_{ \left\{ t^N_i \le t^N_j \right\} } \left(\tilde{F}^N_i \left(t^N_{i-1}, x\left(t^N_{i-1}\right)\right) - F\left(t^N_{i-1}, x\left(t^N_{i-1}\right)\right) \right)	\Delta^N_i }
\end{align*}
constitutes a submartingale[$\cF^N(\,\cdot\,)$]. Indeed, we have for any two $1 \le j \le k \le K_N$
\begin{align*}
	&\expectation{ \norm{ \sum_{i=1}^k \left(\tilde{F}^N_i \left(t^N_{i-1}, x^N_{i-1} \right) - F\left(t^N_{i-1}, x^N_{i-1}\right) \right)	\Delta^N_i} \Bigg|\, \cF^N \left(t^N_j \right) } = \expectation{ \sum_{a=1}^d \left| \sum_{i=1}^k \left(\tilde{F}^N_{i,a} \left(t^N_{i-1}, x^N_{i-1}\right) - F_a \left(t^N_{i-1}, x^N_{i-1}\right) \right)\Delta_i^N \right| \Bigg|\, \cF^N\left(t^N_j \right)}.
\end{align*}
Due to the independence assumption on the stochastic family $\tilde{F}_i^N$, the 1-dimensional processes
\[
	\left\{t^N_0,...,t^N_{K_N} \right\}\ni t^N_j \mapsto \left| \sum_{i=1}^{K_N} 1_{\left \{ t^N_i \le t^N_j \right\} } \left(\tilde{F}^N_{i,a} \left(t^N_{i-1}, x\left(t^N_{i-1}\right)\right) - F_a \left(t^N_{i-1}, x\left(t^N_{i-1}\right)\right) \right) \Delta^N_i \right|
\]
are submartingales[$\cF^N(\,\cdot\,)$] for each $a = 1,\ldots,d$. Summation over the index $a$ proves the statement about the submartingale property. This puts us in position to use Doob's $L^p$-Inequality for equation \eqref{preDoob} with $p=2$ applied to the above submartingale
\begin{align}
		\expectation{\sup_{1\le k \le K_N} \norm{  \hat{x}^N_k - x^N_k }^2} 
		&\le A (\Delta^N)^2 + B \,\expectation{  \sup_{1\le k \le K_N} \left( \norm{ \sum_{i=1}^k \left(\tilde{F}^N_i \left(t^N_{i-1}, x^N_{i-1} \right) - F\left(t^N_{i-1}, x^N_{i-1} \right) \right)	\Delta^N_i		} \right)^2 } \label{preDoob2} \\
		&\le A (\Delta^N)^2 + 4 B \, \expectation{  \left( \, \norm{ \sum_{i=1}^{K_N} \left(\tilde{F}^N_i \left(t^N_{i-1}, x^N_{i-1} \right) - F\left(t^N_{i-1}, x^N_{i-1} \right) \right)	\Delta^N_i	}	\right) ^2 } \label{Doob} \\
				&\le A (\Delta^N)^2 + 4 B d \sum_{a=1}^d \expectation{ \left|\sum_{i=1}^{K_N} \left( \tilde{F}^N_{i,a}\left(t^N_{i-1}, x^N_{i-1} \right) - F_a \left(t^N_{i-1}, x^N_{i-1} \right)	\right) \Delta_i^N	 \right|^2 } \nonumber \\
	&=  A (\Delta^N)^2 + 4 B d \sum_{a=1}^d \sum_{i=1}^{K_N} \left(\Delta^N_i \right)^2 \Var\left[ \tilde{F}^N_{i,a} \left(t^N_{i-1}, x^N_{i-1} \right) \right]\label{IndependenceOfApproximations} \\
		&\le A (\Delta^N)^2 + 4 B T d^2 K_3 \Delta^N. \label{uniformBoundVar}
\end{align}
The first inequality \eqref{preDoob2} follows immediately from inequality \eqref{preDoob}. Inequality \eqref{Doob} stems from Doob's $L^p$-inequality. Equality \eqref{IndependenceOfApproximations} follows from the indepence of the random variables $\tilde{F}^N_{1,a}\left(t^N_{0},	x\left( t^N_{0} \right) \right),\ldots, \tilde{F}^N_{K_N,a}\left(t^N_{K_N-1},	x\left( t^N_{K_N-1} \right) \right)$. The last inequality \eqref{uniformBoundVar} follows from the condition that the variance of the approximation is uniformly bounded. We are now in position to consider the processes $\hat{x}^N$ over the entire interval $[0,T]$. Remember that $\hat{x}^N (t) = \hat{x}^N \left( t^N_{k-1} \right) $ for $t\in \left[t^N_{k-1}, t^N_k \right)$ and $\hat{x}^N(T) = \hat{x}^N \left( t^N_N \right)$, thus,
\begin{align}
		\sup_{t \in [0,T] } \norm{	\hat{x}^N(t) - x(t) }^2 \le 2\left\{	\sup_{1 \le k \le K_N} \norm{ \hat{x}^N \left( t^N_k \right) - x\left( t^N_k \right) }^2 + C^2 \left(\Delta^N \right)^2	\right\}. \label{preL2convergence}
\end{align}
All in all, we find that $\norm{ \sup_{t\in [0,T]} \norm{  \hat{x}^N(t) - x(t) }\, }_{L^2(\p)} \le const\, \left(\Delta^N\right)^{\frac{1}{2}}$ for a sequence of partitions having a mesh $\Delta^N$ which converges to zero. This finishes the proof.
\end{proof}

In addition to the $L^2(\p)$-convergence of the process $\hat{x}^N$, we can state another result on the pathwise convergence for a special choice of the partitioning sequence $\left\{\pi^N:N=1,\ldots,\infty \right\}$. It is an application of Kolmogorov's maximal inequality and follows immediately from the inequality from equation \eqref{preDoob}. We have the following theorem

\begin{theorem}[$a.s.$-convergence of $\hat{x}^N$]\label{preAlmostSureConvergence}
Let $\left\{\pi^N: N=1,\ldots,\infty \right\}$ be a partitioning sequence of the interval $[0, T]$ such that $\sum_{N=1}^{\infty} \Delta^N < \infty$. Then $\sup_{t\in[0,T]} \norm{ \hat{x}^N(t) - x(t) }_1$ converges to zero almost surely.
\end{theorem}
\begin{proof}
Write again $x^N_k := x\left(t^N_k\right)$ and $\hat{x}^N_k := \hat{x}^N\left(t^N_k\right)$ for $k=0,\ldots,K_N$. Consider equation~\eqref{preDoob}, the maximum on the right-hand side can be bounded as
\begin{align*}
\max_{1 \le j \le K_N} \norm{ \sum_{i=1}^j \left( \tilde{F}^N_i \left(t^N_{i-1}, x^N_{i-1} \right)	- F \left(t^N_{i-1}, x^N_{i-1} \right)		\right) \Delta^N_i }_1 \le  \sum_{a=1}^d \left( \max_{1 \le j \le K_N} \left| \sum_{i=1}^j \left( \tilde{F}^N_{i,a} \left(t^N_{i-1}, x^N_{i-1} \right)	- F_a \left(t^N_{i-1}, x^N_{i-1} \right)		\right) \Delta^N_i \right| \right).
\end{align*}
We show that $\max_{1 \le j \le K_N} \left| \sum_{i=1}^j \left( \tilde{F}^N_{i,a} \left(t^N_{i-1}, x^N_{i-1} \right)	- F_a \left(t^N_{i-1}, x^N_{i-1} \right)		\right) \Delta^N_i \right| \rightarrow 0 $ almost surely for each coordinate $a=1,\ldots,d$. An application of Kolmogorov's maximal inequality yields for $\epsilon >0$ that
\[
		\p \left( \max_{1 \le j \le K_N} \left| \sum_{i=1}^j \left( \tilde{F}^N_{i,a} \left(t^N_{i-1}, x^N_{i-1} \right)	- F_a \left(t^N_{i-1}, x^N_{i-1} \right)		\right) \Delta^N_i \right| 	> \epsilon	\right) \le \frac{1}{\epsilon^2} \sum_{i=1}^{K_N} \Var \left( \tilde{F}^N_{i,a} \left(t^N_{i-1}, x^N_{i-1} \right) \Delta_i^N \right) \le \frac{K_3T}{\epsilon^2} \Delta^N.
\]
Hence, we conclude the $a.s.$-convergence from the first Borel-Cantelli Lemma by the convergence assumption on the meshes of partitioning sequence $\left\{\pi^N: N=1,\ldots,\infty \right\}$. The conclusion follows immediately by combining inequality \eqref{preDoob} and \eqref{preL2convergence}, as well as the fact that almost sure convergence is unaffected by continuous transformations.
\end{proof}

\section{Asymptotic Normality of Stochastic Approximation Procedures}\label{AsymptoticNormality}
In this section we prove the asymptotic normality of the stochastic Euler schemes for ODE approximations

\begin{theorem}
Let $\{\pi^N: N\in\N_+ \}$ be the sequence of dyadic partitions of $[0,T]$, i.e. $\pi^N = \{ Tk / 2^N: k=0,1,\ldots,2^N \}$. Let $F=(F_1,\ldots,F_d)$ fulfill the regularity conditions from \eqref{growthTime} and \eqref{growthSpace}. Additionally, let each component of $F$ be continuously differentiable w.r.t. the space coordinate, i.e. $(t,x) \mapsto \nabla_x F_i(t,x)$ is continuous for $i=1,\ldots,d$.\\
Furthermore, let the stochastic approximations $\tilde{F}^N_k$ be regular in that for all $N\in \N_+$ and $k=1,\ldots,2^N$ the $\tilde{F}^N_k$ are independent copies of $\tilde{F}$ where the time-space process $\tilde{F}: [0,T]\times \R^d \times \Omega \rightarrow \R^d$ is Lipschitz-continuous in the space coordinate with the Lipschitz constant $K_2$ as well as continuous in the time coordinate and fulfills the integrability condition
\[ \expectation{ \sup_{t\in [0,T]} \norm{\tilde{F}(t,x(t)) \cdot \tilde{F}(t,x(t))' }_1^2 } < \infty.\]
Then for each $t \in [0,T]$ in the limit	$\lim_{N\rightarrow\infty} \left( \Delta^N \right)^{-\frac{1}{2}} \left( \hat{x}^N (t) - x(t) \right) \sim \mathcal{N} (0, \Sigma(t) )$, where	the function $\Sigma: [0,T]\rightarrow \R^{d\times d}$ is defined as
\[
	\Sigma(t) = \int_0^t P(s,t) \expectation{ (\tilde{F}(s,x(s)) - F(s,x(s))) \cdot (\tilde{F}(s,x(s)) - F(s,x(s)))'} P(s,t)' \,\intd{s}
	\]
and $P$ is the uniform limit of the function $P^N$ on $[0,T]^2$ given by $[0,T]^2 \ni P^N(s,t) = \prod_{ s < t^N_j \le t} \left(I + \Delta^N \nabla_x F \left(t^N_{j-1}, x\left( t^N_{j-1} \right) \right) \right)$.
\end{theorem}

\begin{proof}
We write $\norm{\,\cdot\,}$ throughout the proof for the $2$-norm; since any two norms on the Euclidean space are equivalent, bounds and estimates w.r.t. the 1-norm can be multiplied with the corresponding equivalence constant and are thus valid w.r.t. the 2-norm, too. For a matrix $A$, denote by $\norm{A} := \sup_{x: \norm{x} \le 1} \norm{Ax}$ the spectral norm of $A$. We use the abbreviations
\[
	Z^N:= \left(\Delta^N \right)^{-\frac{1}{2}} \left(\hat{x}^N - x \right) \text{ as well as } x^N_k := x\left(t^N_k \right) \text{ and } \hat{x}^N_k := \hat{x}^N \left(t^N_k \right)
	\]
for simplicity. Choose $t\in [0,T]$ arbitrary but fix, w.l.o.g. $t \in \left[t^N_k,t^N_{k+1} \right)$, if we add the virtual point $t^N_{2^N+1}$ in case that $t=T$. Then
\begin{align*}
	\hat{x}^N(t) - x(t) &= \hat{x}^N_k - x^N_k + \left(x^N_k - x(t) \right) = \hat{x}^N_{k-1} - x^N_{k-1} + \Delta^N \tilde{F}^N_k \left(t^N_{k-1},\hat{x}^N_{k-1} \right) - \int_{t^N_{k-1}}^{t^N_k} F(s,x(s)) \intd{s} - \int_{t^N_k}^t F(s,x(s)) \intd{s}\\
	&= \left[ I + \Delta^N \nabla_x F\left(t^N_{k-1},x^N_{k-1}\right) \right] \left( \hat{x}^N_{k-1} - x^N_{k-1} \right) + \Delta^N \left[ \tilde{F}^N_k\left(t^N_{k-1},x^N_{k-1}\right) - F\left(t^N_{k-1},x^N_{k-1} \right) \right] \\
	&\quad + \Delta^N \left[ \tilde{F}^N_k \left(t^N_{k-1}, \hat{x}^N_{k-1}\right) - F\left(t^N_{k-1}, \hat{x}^N_{k-1}\right) - \left( \tilde{F}^N_{k} \left(t^N_{k-1},x^N_{k-1}\right) - F\left(t^N_{k-1},x^N_{k-1}\right)		\right)			\right] \\
	&\quad + \Delta^N \left[ F\left(t^N_{k-1}, \hat{x}^N_{k-1}\right) - F\left(t^N_{k-1},x^N_{k-1}\right) - \nabla_x F\left(t^N_{k-1},x^N_{k-1}\right) \left(\hat{x}^N_{k-1} - x^N_{k-1} \right) \right] \\
	&\quad + \int_{t^N_{k-1}}^{t^N_k} F\left(t^N_{k-1},x^N_{k-1}\right) - F(s, x(s)) \intd{s}	- \int_{t^N_k}^t F(s,x(s)) \intd{s}
\end{align*}
We make the following definitions
\begin{align*}
	&m^N_k := \tilde{F}^N_k \left(t^N_{k-1},x^N_{k-1}\right) - F\left(t^N_{k-1},x^N_{k-1}\right), \\
	&R^{N,k}_1 :=  \Delta^N \left[ \tilde{F}^N_k \left(t^N_{k-1}, \hat{x}^N_{k-1}\right)  -  \tilde{F}^N_{k} \left(t^N_{k-1},x^N_{k-1}\right) - \left( F\left(t^N_{k-1}, \hat{x}^N_{k-1} \right) - F\left(t^N_{k-1}, x^N_{k-1} \right)			\right) \right], \\
	&R^{N,k}_2 := \Delta^N \left[ \left( F\left(t^N_{k-1}, \hat{x}^N_{k-1} \right) - F\left(t^N_{k-1}, x^N_{k-1} \right)			\right)  -  \nabla_x F\left(t^N_{k-1},x^N_{k-1}\right) \left(\hat{x}^N_{k-1} - x^N_{k-1} \right) \right], \\
	&R^{N,k}_3 := \int_{t^N_{k-1}}^{t^N_k} F\left(t^N_{k-1},x^N_{k-1}\right) - F(s, x(s)) \intd{s} \text{ and } R^{N}_4 := - \int_{t^N_k}^t F(s,x(s)) \intd{s}.
\end{align*}
Set $R^{N,k} :=R^{N,k}_1 + R^{N,k}_2 + R^{N,k}_3$; note that $|R^N_4| \le C \Delta^N$, for a constant $0<C<\infty$. Thus, we get	$Z^N(t) = Z^N ( t^N_k ) + (\Delta^N)^{-\frac{1}{2}} R^N_4$ and at the partitioning points, we face the following structure
\begin{align*}
	Z^N \left(t^N_k \right) = \left[ I + \Delta^N \nabla_x F\left(t^N_{k-1}, x^N_{k-1} \right) \right] Z^N\left(t^N_{k-1}\right) + \left( \Delta^N \right)^{\frac{1}{2}} m^N_k + \left(\Delta^{N}\right)^{-\frac{1}{2}} R^{N,k} \text{ for } 0< t^N_k \le T \text{ and } Z^N(0) = 0.
\end{align*}
Consequently, successive iteration yields
\begin{align}\label{convNormal0}
	Z^N(t) &= \sum_{k: t^N_0 < t^N_k \le t} \left\{ \prod_{j: t^N_k < t^N_j \le t} \left(I + \Delta^N_j \nabla_x F\left(t^N_{j-1}, x^N_{j-1} \right) \right) \right\} \left\{ \left(\Delta^N \right)^{\frac{1}{2}} m^N_k  + \left(\Delta^N \right)^{-\frac{1}{2}} R^{N,k} \right\} + R^N_4. 
\end{align}
In the sequel, we prove that the sum which involves the $m^N_k$ tends to the desired normal distribution, whereas the sum involving the remainder $R^{N,k}$ tends to zero in probability. Hence, $Z^N(t)$ is asymptotically normally distributed with the same parameters. Consider the first sum, we use the definitions
\[
		U_{N,k} :=\prod_{j: t^N_k < t^N_j \le t} \left\{I + \Delta^N_j \nabla_x F\left(t^N_{j-1}, x^N_{j-1} \right) \right\} \left(\Delta^N \right)^{\frac{1}{2}} m^N_k \text{ for } 0< t^N_k \le t
\]
and $U_N := \sum_{ 0< t^N_k \le t} U_{N,k}$. Note that for $N\in\N_+$ the random variables $U_{N,1}, \ldots, U_{N,2^N}$ are independent. W.l.o.g., assume that $\pspace$ is endowed with independent normal distributions $Y_{N,k}$ such that $Y_{N,k} \sim \mathcal{N} \left(0, \text{Cov}\left[ U_{N,k}, U_{N,k} \right]\right)$ for $k=1,\ldots,2^N$. Set $Y_N := \sum_{0<t^N_k\le t} Y_{N,k}$. We prove that the difference of the characteristic functions $\phi_{U_N} - \phi_{Y_N}$ convergences pointwise to zero: For a fix $\alpha\in\R^d$, we show that $\phi_{U_N} (\alpha) - \phi_{Y_N}(\alpha) \rightarrow 0$ as $N \rightarrow \infty$. Therefore, we use the fundamental inequality
\begin{align}\label{condCharFun1}
	\norm{ \phi_{U_N}(\alpha) - \phi_{Y_N}(\alpha) }_{\mathbb{C}} &\le \sum_{k: 0<t^N_k\le t} \norm{ \phi_{U_{N,k}} (\alpha) - \phi_{Y_{N,k}}(\alpha) }_{\mathbb{C}}.
	\end{align}
An application of Lemma~\ref{remainderCharacteristicFunctions} yields 
	\begin{align} \begin{split} \label{preNormal0}
	\eqref{condCharFun1} &\le 2 \norm{\alpha}^2 \sum_{k: 0<t^N_k\le t} \expectation{ \norm{U_{N,k} }^2 1_{ \{ \norm{\alpha}\,\norm{U_{N,k} } > \epsilon/2 \} } } + \epsilon  \norm{\alpha}^2  \sum_{k: 0<t^N_k\le t} \expectation{ \norm{U_{N,k} }^2 } \\
	&\quad +  2 \norm{\alpha}^2 \sum_{k: 0<t^N_k\le t} \expectation{ \norm{Y_{N,k} }^2 1_{ \{ \norm{\alpha}\,\norm{Y_{N,k} } > \epsilon/2 \} } } + \epsilon \norm{\alpha}^2  \sum_{k: 0<t^N_k\le t} \expectation{ \norm{Y_{N,k} }^2 }
\end{split}\end{align}
We show that the first and the third sum of \eqref{remainderCharacteristicFunctions} converge to zero as $N \rightarrow \infty$ for any $\epsilon >0$. This implies that the second and the fourth sum are bounded, and, when multiplied by $\epsilon$, become small, too. We intend to bound $\norm{U_{N,k}}^2$, it is
\begin{align}
		\norm{U_{N,k}} \le \left(\Delta^N\right)^{\frac{1}{2}} \prod_{j: t^N_k < t^N_j \le t} \left(1+\Delta^N \norm{ \nabla_x F(\,\cdot\,,x(\,\cdot\,))}_{\infty} \right) \norm{m^N_k} \le  \left(\Delta^N\right)^{\frac{1}{2}} \exp\left( T  \norm{ \nabla_x F(\,\cdot\,,x(\,\cdot\,))}_{\infty} \right) \norm{m^N_k}. \label{normUnk}
\end{align}
We consider the first sum of \eqref{preNormal0}: Using \eqref{normUnk}, we arrive at
\begin{align}
	 \sum_{k: 0<t^N_k \le t} \expectation{ \norm{U_{N,k}}^2 1_{\left\{ \norm{U_{N,k}} \ge \epsilon \right\} } } 	&=\exp\Big\{ 2 T  \norm{ \nabla_x F(\,\cdot\,,x(\,\cdot\,))}_{\infty} 	\Big\} \, \Delta^N \sum_{k: 0<t^N_k \le t} \expectation{ \norm{m^N_k}^2 1_{\left\{ \norm{m^N_k} \ge  \exp\left(- T \norm{ \nabla_x F(\,\cdot\,,x(\,\cdot\,))}_{\infty} \right) \epsilon \left(\Delta^N\right)^{-\frac{1}{2} }		\right\} } }  \nonumber \\
	&\le c^2 T \expectation{ \sup_{t\in[0,T]} \norm{ \tilde{F}(t,x(t))-F(t,x(t))}^2 1_{ \left\{ \sup_{t\in[0,T]} \norm{ \tilde{F}(t,x(t))-F(t,x(t))} \ge c^{-1} \epsilon \left( \Delta^N\right)^{-\frac{1}{2}} \right\} } }, \label{preNormal1}
\end{align}
where	$c := \exp\left( T \norm{ \nabla_x F(\,\cdot\,,x(\,\cdot\,))}_{\infty} \right) $. An application of Lebesgue's dominated convergence theorem yields that \eqref{preNormal1} converges to zero as $N$ converges to infinity. We obtain for the third sum in equation \eqref{preNormal0}
\begin{align}
	\sum_{k: 0 < t^N_k \le t} \expectation{ \scalar{\alpha}{Y_{N,k}}^2 1_{\left\{  \left| \scalar{\alpha}{Y_{N,k}} \right| \ge \epsilon \right\}}} \le \norm{\alpha}^2 \sum_{k: 0 < t^N_k \le t} \sum_{i=1}^{d} \expectation{ \left|Y_{N,k}^{(i)}\right|^4}^{\frac{1}{2}} \, \p\left(  \left| \scalar{\alpha}{Y_{N,k}} \right| \ge \epsilon 		\right)^{\frac{1}{2}}. \label{preNormal2}
\end{align}
Since, the $d$ elements of the vector $Y_{N,k}$ are normally distributed, we achieve with the notation $\Sigma^{N,k}$ for the covariance matrix $\text{Cov}( U_{N,k}U_{N,k}' )$ that $\sum_{i=1}^d \expectation{ \left|Y_{N,k}^{(i)} \right|^4}^{\frac{1}{2}} = \sum_{i=1}^d \sqrt{3} \Sigma^{N,k}_{i,i} = \sqrt{ 3} \expectation{ \norm{U_{N,k}}^2} \le const\, \Delta^N$, with the help of equation \eqref{normUnk}. In addition, since $\scalar{\alpha}{Y_{N,k}} \sim \mathcal{N} \left(0, \scalar{\alpha}{\Sigma^{N,k} \alpha} \right)$, we get
\[
		\p\left(  \left| \scalar{\alpha}{Y_{N,k}} \right| \ge \epsilon 		\right)^{\frac{1}{2}} \le \sqrt{2} \Phi\left( - \epsilon \scalar{\alpha}{\Sigma^{N,k} \alpha}^{-\frac{1}{2} } \right)^{\frac{1}{2}} \le  \exp\left( - \epsilon^2/4 \norm{\alpha}^{-2} \norm{\Sigma^{N,k}}^{-1} \right),
\]
with the help of a bound given in \cite{chiani2003new}. And $\norm{\Sigma^{N,k}} \le \expectation{ \norm{U_{N,k}}^2 } \le const\, \Delta^N$ from equation \eqref{normUnk}. This proves that \eqref{preNormal2} converges to zero as $N$ tends to infinity. Consequently, $\phi_{U_N} (\alpha) - \phi_{Y_N}(\alpha) \rightarrow 0$, for any $\alpha \in \R^d$.\\
Clearly $Y_N \sim \mathcal{N} ( 0, \Sigma^N )$, where $\Sigma^N = \text{Cov}(U_N, U_N)$ and can be written as 
\begin{align}
	\Sigma^N &= \sum_{k: 0 < t^N_k \le t} \left( \prod_{j: t^N_k < t^N_j \le t} \left(I + \Delta^N_j \nabla_x F\left(t^N_{j-1}, x^N_{j-1}\right) \right) \right) \expectation{ m^N_k \left(m^N_k \right)' } \left(\prod_{j: t^N_k < t^N_j \le t} \left(I + \Delta^N_j \nabla_x F\left(t^N_{j-1}, x^N_{j-1} \right) \right) \right)' \left( t^N_k - t^N_{k-1} \right) \nonumber \\
	&= \sum_{k: 0 < t^N_k \le t} P^N\left( t^N_k, t \right) \expectation{ \left(\tilde{F} \left(t^N_k, x^N_k \right) - F \left(t^N_k, x^N_k \right) \right) \cdot \left(\tilde{F} \left(t^N_k, x^N_k \right) - F \left(t^N_k, x^N_k \right) \right) ' } P^N\left( t^N_k, t \right)'\left( t^N_k - t^N_{k-1} \right) , \label{preNormal3}
\end{align}
with the notation $P^N\left( s, t \right) = \prod_{j: s < t^N_j \le t} \left(I + \Delta^N_j \nabla_x F\left(t^N_{j-1}, x^N_{j-1} \right) \right)$. Due to the continuity of $t\mapsto \nabla_x F(t,x(t))$, we get with the help of Lemma~\ref{matrixProductConv} that $P^N$ converges uniformly on $[0,T]$ to a continuous matrix valued function $P$. An application of Lebesgue's dominated convergence theorem yields that the map which is defined from the factor in the middle of \eqref{preNormal3} as
\[
		\Gamma: [0,T] \rightarrow \R^{d\times d}: \quad t \mapsto \expectation{ \left(\tilde{F}(t,x(t)) - (t,x(t)) \right) \cdot \left(\tilde{F}(t,x(t)) - (t,x(t)) \right)' }
\]
is continuous, thus, $\Sigma^N$ converges uniformly on $[0,T]$ to $\int_0^t P(s,t)  \, \Gamma(s) \,   P(s,t)' \intd{s}$.\\
All in all, $Y_N$ converges to $\cN\left(0, \int_0^t P(s,t) \Gamma(s) P(s,t)'\intd{s} \right)$ in law. It remains to prove that the summed error terms in \eqref{convNormal0} converge to zero in probability. We start with the first error term. Note that due to the independence, we have for all $j \neq k$ that $\expectation{ \scalar{P^N\left(t^N_j,t\right) R^{N,j}_1} {P^N\left(t^N_k,t \right) R^{N,k}_1} } = 0$. Hence, we obtain
\begin{align}
		&\expectation{ \norm{ \left(\Delta^N\right)^{-\frac{1}{2}} \sum_{k: 0 < t^N_k \le t} \prod_{j: t^N_k < t^N_j \le t} \left(I + \Delta^N_j \nabla_x F\left(t^N_j, x^N_j \right) \right) R^{N,k}_1 }^2 } \nonumber \\
		&\le \Delta^N  \exp\left( 2 T \norm{ \nabla_x F(\,\cdot\,,x(\,\cdot\,))}_{\infty} \right) \sum_{k=1}^{2^N} \expectation{ \norm{ \left( \tilde{F}^N_k \left(t^N_k, \hat{x}^N_k \right)  - \tilde{F}^N_k \left(t^N_k, x^N_k\right) \right) -  \left( F^N_k \left(t^N_k, \hat{x}^N_k \right) - F^N_k \left(t^N_k, x^N_k \right) \right) }^2 } \nonumber \\
		&\le const\, \Delta^N  \sum_{k=1}^{2^N} \expectation{ \sup_{t\in[0,T]} \norm{ \hat{x}^N(t) -x(t) }^2 } \rightarrow 0 \text{ as } N \rightarrow \infty. \label{Rem1}
\end{align}
By \eqref{Rem1} the summed first error terms converge to zero in probability. The sum involving the error terms $R^{N,k}_3$ is deterministic and converges to zero: We have $\left( \Delta^N \right)^{- \frac{1}{2}} \sum_{k: 0< t^N_k \le t} \int_{t^N_{k-1}} ^{t^N_k} \norm{ F\left(t^N_{k-1},x^N_{k-1}\right) - F(s,x(s)) } \intd{s} \le const\, \left( \Delta^N \right)^{\frac{1}{2}}$. Finally, consider the sum involving the second error terms $R^{N,k}_2$:
\begin{align}\begin{split}
		\sum_{k: 0 < t^N_k \le t}  \prod_{j: t^N_k < t^N_j \le t} \left(I+\Delta^N \nabla_x F\left(t^N_{j-1}, x^N_{j-1} \right) \right) \left(\Delta^N\right)^{\frac{1}{2}} \left\{  \left( F \left(t^N_{k-1}, \hat{x}^N_{k-1} \right) - F\left(t^N_{k-1}, x^N_{k-1} \right)			\right)  -  \nabla_x F\left(t^N_{k-1}, x^N_{k-1}\right) \left(\hat{x}^N_{k-1} - x^N_{k-1} \right) \right\}. \label{Rem2} 
	\end{split}\end{align}
We apply the mean value theorem to each component $F_i$ of $F$ and get for suitable $\xi^{N,k-1}_i$ between $\hat{x}^N_{k-1}$ and $x^N_{k-1}$ for the norm of equation \eqref{Rem2} the bound
\begin{align}
	 &\le \left( \Delta^N \right)^{\frac{1}{2}} \exp\left(T \norm{ \nabla_x F(\,\cdot\,,x(\,\cdot\,))}_{\infty} \right) \, \sup_{t \in [0,T]} \norm{ \hat{x}^N(t) - x(t)} \sum_{ 0 < t^N_k \le t} \sqrt{ \sum_{i=1}^d \norm{ \nabla_x F_i \left(t^N_{k-1}, \xi^{N,k-1}_i \right) - \nabla_x F_i \left(t^N_{k-1}, x^N_{k-1} \right) }^2 } \nonumber \\
	&\le \left( \Delta^N \right)^{\frac{1}{2}} \exp\left(T \norm{ \nabla_x F(\,\cdot\,,x(\,\cdot\,))}_{\infty} \right) \, \sup_{t \in [0,T]} \norm{ \hat{x}^N(t) - x(t)} \sum_{ 0 < t^N_k \le t} \sum_{i=1}^d \norm{ \nabla_x F_i \left(t^N_{k-1}, \xi^{N,k-1}_i \right) - \nabla_x F_i \left(t^N_{k-1}, x^N_{k-1} \right) } \label{boundNablaF}
\end{align}
Next, define the sets $A^{\rho} := \{y\in\R^d: \norm{y-x(t)} \le \rho \text{ for some } t\in [0,T] \}$. Then the functions $\nabla_x F_i$ are uniformly continuous on $[0,T]\times A^1$. Hence, for every $\epsilon > 0$ there is a $\delta >0$ such that for all $(t,x), (s,y) \in [0,T]\times A^1$ with $\norm{ (t,x) - (s,y)} < \delta$, we have $\max_{i=1,\ldots,d} \norm{ \nabla_x F_i(t,x) - \nabla_x F_i(s,y)} < \epsilon$. Furthermore, due to the Lipschitz-continuity of $F$ all gradients $\nabla_x F_i$ are bounded. Consequently, we obtain for the terms in the sum of equation \eqref{boundNablaF} that
\[ \norm{ \nabla_x F_i \left(t^N_{k-1}, \xi^{N,k-1}_i \right) - \nabla_x F_i \left(t^N_{k-1}, x^N_{k-1} \right) } \le K_4 1_{ \left\{ \sup \left\{ \norm{\hat{x}^N(t) - x(t)}\,: t\in [0,T]  \right\}  > \delta \right\} } + \epsilon 1_{ \left\{ \sup \left\{ \norm{\hat{x}^N(t) - x(t)}\,: t\in [0,T]  \right\}  \le \delta \right\} } \]
for a suitable constant $0 < K_4 < \infty$. This proves that the expectation of \eqref{boundNablaF} can be made arbitrarily small depending on the choice of $\epsilon > 0$ as $N$ converges to infinity.
Thus, the summed third error converges to zero in distribution. Hence, the overall error converges in distribution to the constant zero. Theorem 2.7 of \cite{van1998asymptotic} states that for random variables $X_n$ and $Y_n$ such that $X_n \Rightarrow X$ and $Y_n \Rightarrow c$ where $c$ is a constant, we have $(X_n, Y_n) \Rightarrow (X,c)$. This yields the desired asymptotic normality.
\end{proof}

\appendix
\section{Deferred Proofs and Background Material}\label{Appendix}

\begin{lemma}[Discrete Gronwall Inequality]\label{discreteGronwallInequality}
Let $\{f_k: k\in\N \}, \{g_k: k\in\N \}, \{y_k: k\in\N\}$ be positive sequences in $\Rplus$ which fulfill $y_n \le f_n + \sum_{k=0}^{n-1} g_k y_k$ for every $n\in \N$. Then, we have $y_n  \le f_n + \sum_{k=0}^{n-1} f_k g_k \prod_{j=k+1}^{n-1} (1+g_j) \le f_n + \sum_{k=0}^{n-1} f_k g_k \exp\left( \sum_{j=k+1}^{n-1} g_j \right)$ for each $n\in\N$. The first inequality is actually sharp.
\end{lemma}

\begin{lemma}[Estimates for characteristic functions]\label{remainderCharacteristicFunctions}
Let $X$ be a $d$-dimensional real random variable on $\pspace$ and let $\mathcal{G} \subseteq \cA$ be a sub-$\sigma$-algebra of $\cA$. Define $\mu = \expectation{X \,|\, \cG}$ and $\Sigma := \expectation{ X X'\,|\, \cG}$. Then for the conditional characteristic function of $X$ w.r.t. $\cG$, $\phi_X|_{\mathcal{G}}$, it holds that for each $t\in \R^d$
\begin{align*}
	\phi_X|_{\mathcal{G}} (t) = \expectation{ e^{i\langle t, X \rangle } \,\big|\, \cG }= \left(1 - \frac{1}{2} \langle t, \Sigma t \rangle  \right) e^{i \scalar{t}{\mu} }+ r(t),
\end{align*}
where the remainder can be bounded as follows
\[
	|r(t)| \le 2 \norm{t}^2 \expectation{ \norm{X}^2 1_{ \{ \norm{t}\,\norm{X} > \epsilon/2 \} } \,|\, \cG } + \epsilon \norm{t}^2  \expectation{ \norm{X}^2 1_{ \{ \norm{t}\,\norm{X} \le \epsilon \} } \,|\, \cG }.
\]
\end{lemma}

\begin{proof}
We can decompose the conditional characteristic function in a real and an imaginary function
\[
		\phi_{X|\cG}(t) = \expectation{ \cos \langle t, X \rangle \,|\, \cG} + i\, \expectation{ \sin \langle t, X \rangle \,|\, \cG}
\]
from which we can compute the gradients and the Hessian matrices. We get for the gradient of the $\sin$- and $\cos$-term
\[
		t \mapsto - \expectation{ \sin\langle t, X \rangle \cdot X \,|\, \cG} \text{ and } t \mapsto  \expectation{ \cos\langle t, X \rangle \cdot X \,|\, \cG}.
\]
The Hessian matrices are given by
\[
		t \mapsto - \expectation{ X \cdot \cos\langle t, X \rangle \cdot X' } \text{ and by } t \mapsto - \expectation{ X \cdot \sin\langle t, X \rangle \cdot X' }. 
\]
In the first place let $\expectation{ X \,|\, \cG} = 0$ a.s.[$\p$], then
\begin{align*}
		\expectation{ \cos \langle t, X \rangle \,|\, \cG} &= 1 - \frac{1}{2} \cdot t' \cdot \expectation{ X\cdot X'\,|\, \cG} \cdot t - \frac{1}{2} \cdot t' \cdot \expectation{ X\cdot( \cos \langle \theta_1 t, X\rangle -1 ) \cdot X' \,| \, \cG} \cdot t,\\
		\expectation{ \sin \langle t, X \rangle \,|\, \cG} &=  - \frac{1}{2} \cdot t' \cdot \expectation{ X \cdot \sin \langle \theta_2 t, X\rangle \cdot X' \,| \, \cG} \cdot t,
\end{align*}
for suitable $\theta_1,\theta_2 \in [0,1]$ by the mean value theorem. Hence, $\phi_{X|\cG} (t) = 1 - \frac{1}{2} \langle t, \Sigma t \rangle + r(t)$, where
\begin{align*}
		|r(t) | &= \frac{1}{2} \left| \left\langle t, \left\{ \expectation{ X\cdot \Big( \cos \langle \theta_1 t, X\rangle -1 \Big) \cdot X' \,| \, \cG} +  i \expectation{ X\cdot \sin \langle \theta_2 t, X\rangle \cdot X' \,| \, \cG} \right\} t \right \rangle	\right| \\
		&\le \frac{1}{2} \sup_{\theta_1 \in [0,1]} \left| \left \langle t, \expectation{ X\cdot \Big( \exp( i \langle \theta_1 t, X\rangle ) - 1 \Big) \cdot X' \,| \, \cG} t \right\rangle \right| \\
		&\quad  + \frac{1}{2} \sup_{\theta_1, \theta_2 \in [0,1]} \left | \left \langle t, \expectation{ X\cdot \Big(\sin \langle \theta_2 t, X\rangle -  \sin \langle \theta_1 t, X\rangle \Big) \cdot X' \,| \, \cG}  t \right \rangle	\right|
\end{align*}
Next, we make use of the estimate $|1-e^{i \alpha} | \le \min\{|\alpha|, 2\}$ real $\alpha$ for the first term. For the second term, we use that the real sinus function is Lipschitz continuous with Lipschitz-constant 1, i.e. $|\sin x - \sin y| \le |x-y|$. Hence, we have the following two estimates,
\begin{align*}
		 \left| \left \langle t, \expectation{ X\cdot \Big( \exp( i \langle \theta_1 t, X\rangle ) - 1 \Big) \cdot X' \,| \, \cG} t \right\rangle \right| & \le \norm{t}^2 \expectation{ \norm{X}^2 |\exp( i \scalar{ \theta_1 t}{X}) - 1| \,|\, \cG } \\
		&\le 2 \norm{t}^2 \expectation{ \norm{X}^2 1_{ \{ \norm{t}\,\norm{X} > \epsilon\} } \,|\, \cG } + \epsilon \norm{t}^2  \expectation{ \norm{X}^2 1_{ \{ \norm{t}\,\norm{X} \le \epsilon\} } \,|\, \cG }.
		\end{align*}
		And,
		\begin{align*}
		&\left | \left \langle t, \expectation{ X\cdot \Big(\sin \langle \theta_2 t, X\rangle -  \sin \langle \theta_1 t, X\rangle \Big) \cdot X' \,| \, \cG}  t \right \rangle	\right| \\
		&\qquad\qquad\qquad\qquad \le 2 \norm{t}^2 \expectation{ \norm{X}^2 1_{ \{ \norm{t}\,\norm{X} > \epsilon/2 \} } \,|\, \cG } + \epsilon \norm{t}^2  \expectation{ \norm{X}^2 1_{ \{ \norm{t}\,\norm{X} \le \epsilon/2 \} } \,|\, \cG }. 
\end{align*}
Combining these estimates, the remainder can be bounded as claimed. For general $\mu $, we find that	$\phi_{X|\cG} (t) = \phi_{X-\mu} (t) e^{i \scalar{t}{\mu} }$, hence, we can apply the above analysis once again. This finishes the proof.
\end{proof}

\begin{lemma}\label{matrixProductConv}
Let $\pi^N$, $N\in\N_+$, be the sequence of dyadic partitions of $[0,T]$, $T>0$, such that $0 = \tau^N_0 < \tau^N_1 < ... < \tau^N_{2^N} = T$. Let $\chi$ be a continuous matrix valued mapping from $[0,T]$ to $\R^{d\times d}$. Let $||\cdot||$ be a submultiplicative matrix norm on  $\R^{d\times d}$. Then there is a continuous map
\[
		P: [0,T]^2 \rightarrow \R^{d\times d} \text{ such that } \sup_{ \substack{s,t\in [0,T],\\ s \le t}} \norm{ \prod_{s < t^N_i \le t} \left(I + \Delta^N_i \chi\left( t^N_{i-1} \right) \right) - P(s,t) } \rightarrow 0 \text{ as } N \rightarrow \infty.
\]
\end{lemma}
\begin{proof}
For $0\le s\le t \le T$ write $P^N(s,t) := \prod_{s < t^N_i \le t} \left(I + \Delta^N_i \chi\left( t^N_{i-1} \right) \right)$. We show that the $(P^N)_{N \in \N_+}$ are Cauchy w.r.t. $\norm{\,\cdot\,}_{\infty}$ on $[0,T]$. Fix some $0 \le s \le t \le T$. Let $M \ge N$. Then for a $k \le 2^N$, we have $P^N(s,t) = D_0\cdot D_1\cdot\ldots\cdot D_k $, where each $D_i = I + \Delta^N_i \chi\left( t^N_{i-1} \right)$ for some $t^N_i \in (s,t]$ and $D_0$ is determined by the unique $t^N_l$ which fulfills $t^N_{l-1} \le s < t^N_l \le t$ . Since $\pi^N \subseteq \pi^M$, we can choose for $i=1,\ldots,k$ the factors
\[
	F_i = \left(I + \Delta^M_u \chi\left( t^M_{u-1} \right) \right) \cdot \ldots \cdot \left( I + \Delta^M_{u+v} \, \chi\left( t^M_{u+v-1} \right) \right) \text{ such that for } i \neq l: \; \left(t^N_{i-1}, t^N_i \right] = \left(t^M_{u-1}, t^M_u \right]\cup\ldots\cup \left(t^M_{u+v-1}, t^M_{u+v} \right]
\]
for $v = 2^{M-N}-1$. Furthermore, there is a factor $F_0$ given by
\[
	F_0 =  \left(I + \Delta^M_l \chi\left( t^M_{l-1} \right) \right) \cdot \ldots \cdot \left( I + \Delta^M_{l+\tilde{v} } \, \chi\left( t^M_{l+\tilde{v}-1} \right) \right)
	\]
for the unique $t^M_l$ which fulfills $t^M_{l-1} \le s < t^M_l \le t$ and $\tilde{v} \le 2^{M-N}$. Then $P^M(s,t) = F_0\cdot F_1\cdot\ldots\cdot F_k \cdot \text{Res}$, where the residual factors of $P^M(s,t)$ are collected in $\text{Res}$. Hence, we can write
\begin{align}
		P^M(s,t) - P^N(s,t) &= F_0\cdot F_1\cdot\ldots\cdot F_k \cdot \text{Res} - D_0\cdot D_1\cdot \ldots\cdot D_k \nonumber \\
		\begin{split}
		&= (F_0 - I) \cdot F_1 \cdot \ldots \cdot F_k \cdot \text{Res} +  F_1 \cdot \ldots \cdot F_k \cdot (\text{Res} - I)\\
		&\quad+\big\{ F_1 \cdot\ldots\cdot F_k - D_1 \cdot \ldots\cdot D_k \big\} - (D_0 - I) \cdot D_1\cdot \ldots \cdot D_k.\label{D-M2} \end{split}
\end{align}
Firstly, we have $\norm{\text{Res} } \le \exp \left(\Delta^N \norm{\chi}_{\infty} \right)$, as well as, $\max\big\{ \norm{F_1\cdot,\ldots\cdot F_k},\; \norm{D_1\cdot,\ldots\cdot D_k} \big\} \le \exp \left(T \norm{\chi}_{\infty} \right)$ and secondly,
\[
	\max\big\{ \norm{ F_0 - I},\; \norm{D_0 - I},\; \norm{\text{Res} - I} \big\} \le \Delta^N \norm{\chi}_{\infty} \exp\left( \Delta^N\norm{\chi}_{\infty} \right).
\]
Thirdly, we can write the main term as $F_1\cdot\ldots\cdot F_k  - D_1\cdot \ldots\cdot D_k = \sum_{j=1}^k\; \prod_{i=1}^{j} F_i \prod_{i=j+1}^k D_i - \sum_{j=0}^{k-1}\; \prod_{i=1}^{j} F_i \prod_{i=j+1}^k D_i$, which implies for the norm of this term
\begin{align}\label{D-F}
		\norm{ F_1\cdot\ldots\cdot F_k - D_1\cdot\ldots\cdot D_k } \le \sum_{j=1}^{k} \max_{1\le i \le k} \norm{F_i} ^{j-1} \norm{ F_j - D_j}  \max_{1\le i \le k} \norm{D_i}^{k-j}.
\end{align}
The factors of each summand in \eqref{D-F} can be estimated as follows, we have for the first and the third factor
\[
		\norm{D_i} \le 1+\Delta^N \norm{\chi}_{\infty} \le \exp \left( \Delta^N \norm{\chi}_{\infty} \right) \text{ as well as } \norm{F_i} \le (1+\Delta^M \norm{\chi}_{\infty} ) ^{2^{M-N}} \le \exp\left( \Delta^N \norm{\chi}_{\infty} \right).
\]
For the factor in the middle, we use the definition of the modulus of continuity: Define for $\delta > 0$ the function $w(\delta,\chi) := \sup\{ ||\chi(s) - \chi(t)||: |s-t|\le \delta \}$. Then
\begin{align*}
		\norm{ F_j - D_j} &\le \Delta^N w(\Delta^N, \chi) + \sum_{j=2}^{2^{M-N}} \binom{2^{M-N}}{j} \left( \Delta^M \norm{\chi}_{\infty} \right)^j  \le \Delta^N w(\Delta^N,\chi) + \left( \Delta^N \right)^2 \norm{\chi}_{\infty}^2 \exp \left(\Delta^N \norm{\chi}_{\infty} \right).
\end{align*}
Eventually, we combine these estimates to find that \eqref{D-M2} can be bounded over all $s$ and $t$ by $c\, \left(\Delta^N + w(\Delta^N, \chi) \right)$, for a suitable constant $0<c<\infty$ which does not depend on $N$. Since $\chi$ is uniformly continuous on $[0,T]$, we have $\lim_{N\rightarrow\infty} w(\Delta^N,\chi)=0$. This proves the Cauchy property and consequently the uniform convergence of the sequence $(P^N)_{N\in\N_+}$ to a limit function $P$. It remains to prove the continuity of this $P$. We have for all $(s,t), (s_0,t_0)$ in $[0,T]^2$ that
\begin{align}\begin{split}\label{normP0}
	\norm{ P(s,t) - P(s_0,t_0) } &\le \norm{P(s,t) - P^N(s,t)} + \norm{P^N(s,t) - P^N(s_0,t)} \\
	&\quad + \norm{P^N(s_0,t) - P^N(s_0,t_0)} + \norm{P^N(s_0,t_0) - P(s_0,t_0)}.
\end{split}\end{align}
And we can compute, that for $0\le a < b < c \le T$, we have
\[
		\norm{ \prod_{a < t^N_i \le b} \left( I + \Delta^N \chi\left( t^N_i \right) \right) - \prod_{a < t^N_i \le c} \left( I + \Delta^N \chi\left( t^N_i \right) \right) } \le \norm{\chi}_{\infty} \exp\Big( T\norm{\chi}_{\infty} \Big) \exp\left( (c-b + \Delta^N) \norm{\chi}_{\infty} \right) \left( (c-b) + \Delta^N \right)
\]
where the terms involving the $\Delta^N$ stem from the fact that $P^N$ is discontinuous at the partitioning points. All in all, the remaining terms in equation \eqref{normP0} can be bounded with
\begin{align*}
		&\max\left\{\norm{P^N(s,t) - P^N(s_0,t) },\; \norm{P^N(s_0,t) - P^N(s_0,t_0) } \right\} \le const\, \Big( |\max(|s-s_0|, |t-t_0|) + \Delta^N \Big)
\end{align*}
This yields the desired continuity of the limit function $P$.
\end{proof}

\bibliographystyle{plainnat} 
\bibliography{Bibliography}

\end{document}